\documentclass[reqno,12pt,a4paper]{amsart}

\usepackage[utf8]{inputenc}
\usepackage{enumerate}
\usepackage{amstext,amsmath,amsthm,amsfonts,amssymb,amscd}
\usepackage{latexsym,mathrsfs,dsfont,euscript}

\usepackage{fullpage}

\usepackage{xcolor}

\usepackage{hyperref} 
\hypersetup{
  colorlinks,%
  citecolor=blue,%
    filecolor=red,%
    linkcolor=red,%
    urlcolor=blue
}

\usepackage{graphicx}
\usepackage{caption}
    \usepackage{subcaption}

\usepackage{listings}

\definecolor{codegreen}{rgb}{0,0.6,0}
\definecolor{codegray}{rgb}{0.5,0.5,0.5}
\definecolor{codepurple}{rgb}{0.58,0,0.82}
\definecolor{backcolour}{rgb}{1,1,1}

\lstdefinestyle{mystyle}{
    backgroundcolor=\color{backcolour},
    commentstyle=\color{codegreen},
    keywordstyle=\color{magenta},
    numberstyle=\tiny\color{codegray},
    stringstyle=\color{codepurple},
    basicstyle=\ttfamily\footnotesize,
    breakatwhitespace=false,
    breaklines=true,
    captionpos=b,
    keepspaces=true,
    numbers=none,
    numbersep=5pt,
    showspaces=false,
    showstringspaces=false,
    showtabs=false,
    tabsize=2
}

\newtheorem{theorem}{Theorem}
\newtheorem{proposition}[theorem]{Proposition}
\newtheorem{lemma}[theorem]{Lemma}

\theoremstyle{definition}

\theoremstyle{remark}

\numberwithin{equation}{section}

\newcommand{\abs}[1]{\left\vert#1\right\vert}

\allowdisplaybreaks

\title{On Frink's type metrization of weighted graphs}

\author[]{Mar\'ia Florencia Acosta}
\email{mfacosta@santafe-conicet.gov.ar}
\author[]{Hugo Aimar}
\email{haimar@santafe-conicet.gov.ar}
\author[]{Ivana G\'{o}mez}
\email{ivanagomez@santafe-conicet.gov.ar}
\thanks{This work was supported by the Ministerio de Ciencia, Tecnolog\'ia e Innovaci\'on-MINCYT in Argentina: CONICET and ANPCyT; and the UNL}

\begin{document}

\begin{abstract}
Using the technique of the metrization theorem of uniformities with countable bases, in this note we provide, test and compare an explicit algorithm to produce a metric $d(x,y)$ between the vertices $x$ and $y$ of an affinity weighted undirected graph.
\end{abstract}

\maketitle

\section{Introduction}

The construction of metrics in data sets is a problem of current interest in data analysis. Of course the
metrics built on a given data set should reflect, in a quantitative form, the affinity of the different data points. There are many reasons for the search of such metric structures on data sets. In particular adequate metrics provide notions of neighborhood of a given point which are not provided a priori directly by the affinity. But more important is the fact that in metric spaces many of the properties of Euclidean spaces still hold and covering and partitions can be done with a metric control which is natural for each setting.

Perhaps the best known metrization method is that of diffusive metrics due to Coifman and Laffon \cite{CoifmanLafon06}. Once a Laplace type operator is built from the affinity matrix between data, the spectral analysis of this operator provides a diffusion kernel which gives a family of metrics on the data set at different times. The size of the eigenvalues allows the detection of the main features of and hence the approximation of a high
dimensional space by another space with small dimension. In pure mathematics the problem of metrization of general topological spaces is old and well known. In particular, the metrization of the topology induced on a set $X$ by a uniformity on $X\times X$ was considered and solved in \cite{Frink37}, see also \cite{Chi27} and \cite{Kelleybook75} when the uniform structure has a countable basis. The result is that a topology induced by a uniform structure is  metrizable if an only if the uniformity has a countable basis. Even when so stated the results seems to have a qualitative character its proof entails a quantitative lemma due to Frink that allows to obtain a metric from the affinity going through the uniform structure induced by the affinity between the data points.

The first use of this quantitative lemma is due to Macias and Segovia (\cite{MaSe79Lip}) in order to show that quasi-distances are equivalent to powers of metrics. In \cite{AiGoAffinity} sufficient conditions on a general affinity kernel $K$ on an abstract set $X$ are given in order to  obtain a Newton type potential form for $K$ in terms of a natural metric on $X$. Loosely speaking \cite{AiGoAffinity} shows that, with a quantitative transivity hypothesis, we have that $K(x,y)=\varphi(d(x,y))$ for some ``metric'' $d$ and some quasi-convex decreasing function $\varphi$ defined on the positive real numbers.

In this note we aim to provide, test and compare an explicit algorithm in order to obtain a metric type function $d(x,y)$ between the vertices $x$ and $y$ associated to an affinity weighted graph. The algorithm gives actually a uniform family of metrics that provide together a profuse enough family of balls.

The second section of this note is devoted to state and prove the main result as a consequence of Frink's Lemma as stated and proved in \cite{Kelleybook75}. Section~\ref{sec:algorithm} describes the algorithm for the case of finite $X$. In Section~\ref{sec:TestComparison} we test and compare the algorithm in some special weighted graphs

\section{Pseudometrization of affinity kernels and weighted undirected graphs through Frink's Lemma}\label{sec:pseudometrization}

Even when the problem is motivated by the finite setting provided by weighted graphs, the basic theory does not need any assumption regarding cardinality. Hence, in this section, we assume that $X$ is a set and $K: X\times X\to [0,\infty)$ is a nonnegative function such that for $x$ and $y$ in $X$, $K(x,y)$ is a measure of affinity between $x$ and $y$.

A pseudo-metric on the set $X$ is a function $d:X\times X\to [0,\infty)$ such that
\begin{itemize}
\item[(p-m.1)] $d(x,x)=0$ for every $x\in X$;
\item[(p-m.2)] $d(x,y)=d(y,x)$, $x, y\in X$;
\item[(p-m.3)] $d(x,z)\leq d(x,y)+d(y,z)$ for every $x, y, z\in X$.
\end{itemize}
A pseudo-metric is a metric if $d(x,y)=0$ only when $x=y$.

Let us now proceed to state Frink's Lemma as given in Chapter~6 of Kelley's book \cite{Kelleybook75}. Some notation to simplify further statements is in order. With $\triangle$ we denote the diagonal of $X\times X$. In other words $\triangle=\{(x,x): x\in X\}$. Given a subset $U$ of $X\times X$ we write $U^{-1}$ to denote the set $\{(x,y)\in X\times X: (y,x)\in U\}$. We say that $U$ is symmetric if $U=U^{-1}$. Given two subsets $U$ and $V$ of $X\times X$, the composition is defined by $V\circ U = \{(x,z)\in X\times X: \textrm{there exist } y\in X \textrm{ with } (x,y)\in U \textrm{ and } (y,z)\in V \}$.

\begin{lemma}\label{lemma:Frink}
Let $X$ be a set and let $\{U_m: m=0,1,2,\ldots\}$ be a sequence of subsets of $X\times X$ satisfying the following properties
\begin{enumerate}[i)]
\item $U_0 = X\times X$;
\item $U_n=U_n^{-1}$ for every $n$;
\item $\triangle\subset U_n$ for every $n$;
\item $U_{n+1}\circ U_{n+1}\circ U_{n+1}\subseteq U_n$ for every $n$.
\end{enumerate}
Then, there exist a pseudo-metric $d$ defined on $X$ such that for every $n=1,2,3,\ldots$
\begin{equation*}
U_n\subset \{(x,y)\in X\times X: d(x,y)<2^{-n}\}\subset U_{n+1}.
\end{equation*}
\end{lemma}

The above control of the given sequence $\{U_n: n=0,1,2,\ldots\}$ by the level sets of the pseudo-metric $d$ seems to be of qualitative character. Nevertheless, when the sequence $U_n$ is itself given by level sets of some function $K$ on $X\times X$, this control becomes quantitative and allows to find a natural notion of distance provided by $K$.

In the sequel, for a given subset $V$ of $X\times X$ we shall use $V^{(n)}$ to denote the composition $V\circ V\circ V\ldots\circ V$ $n$ times.

Let us now prove that under some mild conditions in $K$ it is possible to construct increasing sequences $\{\lambda(k): k=0,1,2,\ldots\}$ such that $U_{k+1}\circ U_{k+1}\circ U_{k+1}\subseteq U_k$ whenever $U_k=\{K>\lambda(k)\}$.
\begin{lemma}\label{lemma:sequenceLambda}
Let $X$ be a set and let $K$ be a nonnegative symmetric real function defined on $X\times X$ satisfying
\begin{enumerate}[a)]
\item $K(x,x)=\sup_{y\in X} K(x,y)$ for every $x\in X$;
\item $0<\Lambda_\infty=\sup\{\alpha>0:\, \{K>\alpha\}^{(m)}=X\times X \textrm{ for some integer } m\}\leq \infty$.
\end{enumerate}
Then, for every $\Lambda$ with $0<\Lambda<\Lambda_\infty$ there exists a finite sequence $0=\lambda(0)<\lambda(1)<\ldots<\lambda(k)=\Lambda$ such that
$\{K>\lambda(i)\}^{(3)}\subseteq \{K>\lambda(i-1)\}$ for every $i=1,2,\ldots,k$. Moreover,
$\triangle\subset \{K>\lambda(i)\}$ for every $i=0,1,2,\ldots,k$.
\end{lemma}
\begin{proof}
Let us first notice that the set \(\displaystyle{A=\{\alpha>0:\{K>\alpha\}^{(m)}=X\times X \textrm{ for some in\-te\-ger\,} m\}}\) is an interval or the whole half line $\mathbb{R}^+$. This fact follows from the monotonicity of the level sets of $K$. In other words if $\alpha\in A$ and $0<\beta<\alpha$ then $\{K>\beta\}\supset\{K>\alpha\}$, so that $\{K>\beta\}^{(m)}\supset\{K>\alpha\}^{(m)}=X\times X$ and $\beta\in A$. On the other hand, for each $\alpha\in A$ we have that $\triangle\subset \{K>\alpha\}$. This follows from property $a)$ of the kernel $K$. In fact, if for some $x_0\in X$ we have $K(x_0,x_0)\leq\alpha$, then $\sup_{y\in X} K(x_0,y)\leq\alpha$ and for no $m\in\mathbb{N}$ the point $(x_0,x_0)$ would belong to $\{K>\alpha\}$. But since $\alpha\in A$, for some $m$, $\{K>\alpha\}^{(m)}=X\times X\supset\{(x_0,x_0)\}$.

Let us pick $0<\Lambda<\Lambda_\infty$. From the above remarks, we have that $\Lambda\in A$ and $\triangle\subset\{K>\Lambda\}$. Set $m_\Lambda=\min \{m\in\mathbb{N}: \{K>\Lambda\}^{(m)}=X\times X \}$. In other words, $\{K>\Lambda\}^{(m_\Lambda)}=X\times X$ but $\{K>\Lambda\}^{(m_\Lambda -1)}\subsetneqq X\times X$. We may assume that $m_{\Lambda}\geq 3$. Now, consider the set $A_1=\{\alpha>0: \{K>\Lambda\}^{(3)}\subseteq \{K>\alpha\} \}$. If $A_1=\emptyset$, the sequence that we are looking for has only two elements $\lambda(0)=0$ and $\lambda(1)=\Lambda$. And the desired inclusion $\{K>\lambda(1)\}^{(3)}\subseteq X\times X=\{K>\lambda(0)\}$ holds trivially. If $A_1\neq \emptyset$ take $\Lambda_1\in A_1$ with $\Lambda_1 > \sup A_1 -\varepsilon$ for some fixed as small as desired and positive $\varepsilon$. Set now $A_2=\{\alpha>0: \{K>\Lambda_1\}^{(3)}\subseteq \{K>\alpha\}\}$. If $A_2=\emptyset$, then we are done with $\lambda(0)=0$, $\lambda(1)=\Lambda_1$ and $\lambda(2)=\Lambda$. So may keep iterating this selection process by choosing $\lambda_i\in A_i=\{\alpha>0:\{K>\Lambda_{i-1}\}^{(3)}\subseteq\{K>\alpha\}\}$ with $\Lambda_i>\sup A_i -\varepsilon$. Since for $\{K>\Lambda\}^{(m_\Lambda)}=X\times X$, after at most the integer part of $m_\Lambda/3$ plus one iterations the process stops providing a finite sequence of levels $\Lambda_0:=\Lambda>\Lambda_1>\Lambda_2>\ldots >\Lambda_k$. Taking $\lambda(i)=\Lambda_{k-i}$ for $i=0,1,\ldots,k$ we get the desired result.
\end{proof}

Let us point out that for discrete settings or for continuous kernels $K$ the choice of the sequence $\Lambda_i$ in the argument above can be accomplished by taking the maximum of each $A_i$.
Hence the $\varepsilon$-approximation argument is not necessary. From the above two lemmas we are in position to state and prove the main results of this section.
\begin{theorem}
Let $X$ be a set. Let $K$ be a nonnegative symmetric function defined on $X\times X$ satisfying $a)$ and $b)$ in Lemma~\ref{lemma:sequenceLambda}. Then for every sequence $\lambda=\{\lambda(i): i=0,1,\ldots,k=k(\lambda)\}$ as in Lemma~\ref{lemma:sequenceLambda}, there exists a pseudo-metric $d_\lambda$ defined on $X$ such that
\begin{enumerate}[1)]
\item $\{K>\lambda(i)\}\subseteq \{d_\lambda<2^{-i}\}\subseteq \{K>\lambda(i-1)\}$ for every $i=1,2,\ldots,k$;

\item the function
\begin{equation*}
\delta_\lambda = 2^{-\lambda^{-1}\circ K},
\end{equation*}
with $\lambda^{-1}$ the inverse of any increasing extension of $\lambda(i)$ to the whole interval $[0,k(\lambda)]$, is equivalent to the pseudo-metric $d_\lambda$ with constants that are uniform in $\lambda$. Precisely,
\begin{equation*}
\frac{\delta_\lambda (x,y)}{4} < d_\lambda(x,y)\leq 2 d_\lambda(x,y).
\end{equation*}
\end{enumerate}
\end{theorem}
\begin{proof}
From Lemma~\ref{lemma:sequenceLambda} the sequence $U_i=\{K>\lambda(i)\}$ satisfies $i)$ to $iv)$ of Lemma~\ref{lemma:Frink}. Hence there exists a pseudo-metric $d_\lambda$ defined on $X$ such that $1)$ holds. In order to prove $2)$ take $(x,y)\in X\times X$ such that $d_\lambda(x,y)>0$. Hence for some $i=0,1,\ldots,k(\lambda)$ we have
\begin{equation*}
2^{-(i+1)}\leq d_\lambda(x,y)<2^{-i}.
\end{equation*}
The inequality $d_\lambda(x,y)<2^{-i}$ and the second inclusion in $1)$ shows that $K(x,y)>\lambda(i-1)$. The inequality $2^{-(i+1)}\leq d_\lambda(x,y)$ and the first inclusion in $1)$ shows that $K(x,y)\leq\lambda(i+1)$. If $\lambda$ is any strictly increasing extension of the sequence $\lambda(i)$ for $i=0,\ldots,k$ to the interval $[0,k]$ and $\lambda^{-1}$ denote its inverse function, we have that $2^{-(i+1)}\leq d_\lambda(x,y)<2^{-i}$, and
\begin{equation*}
i-1<(\lambda^{-1}\circ K)(x,y)\leq i+1.
\end{equation*}
From this inequalities it readily follows that $\delta_\lambda= 2^{-\lambda^{-1}\circ K}$ is equivalent to $d_\lambda$. In fact,
\begin{equation*}
\frac{1}{4} = 2^{-(i+1)} 2^{i-1} < d_\lambda(x,y) 2^{(\lambda^{-1}\circ K)(x,y)}\leq 2^{-i}2^{i+1}=2.
\end{equation*}
\end{proof}

Let us point out that the function $\delta_\lambda$ in the above result satisfies a triangle type inequality with triangular constant equal to $8$ no matter what the kernel $K$ or the sequence $\lambda$, satisfying Lemma~\ref{lemma:sequenceLambda}, are. In fact,
\begin{equation*}
\delta_\lambda(x,z)\leq 4 d_\lambda(x,z)\leq 4 (d_\lambda(x,y)+d_\lambda(y,z))\leq 8(\delta_\lambda(x,y)+\delta_\lambda(y,z))
\end{equation*}
for every $x$, $y$ and $z\in X$.

Regarding the extension of $\lambda$ in order to produce the function $\lambda^{-1}$ needed to explicitly give the quasi-metric $\delta_\lambda$, let us  observe that two extremal cases can be explicitly given. In fact, let $\overline{\lambda}^{-1}:[0,\lambda(k)]\to [0,k]$ with $\overline{\lambda}^{-1}(t)=i$ for $\overline{\lambda}(i-1)<t\leq\overline{\lambda}(i)$ and $i=1,\ldots,k$. Also $\overline{\lambda}^{-1}(0)=0$. Another possible $\lambda^{-1}$ is a lower case $\underline{\lambda}^{-1}:[0,\lambda(k)]\to [0,k-1]$ given by
$\underline{\lambda}^{-1}(t)=i-1$ for $\underline{\lambda}(i-1)<t\leq \underline{\lambda}(i)$ for $i=1,\ldots,k$.

It is also worth noticing that Frink's metric and hence also $\delta_\lambda$, do not reflect the scaling factor associated to the choice of $\Lambda$ in Lemma~\ref{lemma:sequenceLambda}. This is due to the fact that Frink's metric $d_\lambda$ takes only values between zero and one. So that, being $\delta_\lambda$  equivalent to $d_\lambda$, also our quasi-metric $\delta_\lambda$ is bounded.

The sequence $\lambda(i)$ contains also the information of a family of $\delta_\lambda$ balls defined directly as level sets of the affinity kernel $K$.

\begin{proposition}\label{propo:deltalambdaBalls}
For $0<r<1$ we have that the open $\delta_\lambda$ ball centered at $x\in X$ with radious $r$, is given by
\begin{equation*}
B_{\delta_\lambda}(x,r)=\{y\in X: K(x,y)>\lambda(\log_2\tfrac{1}{r})\}.
\end{equation*}
\end{proposition}
\begin{proof}
The inequality $K(x,y)>\lambda(\log_2\tfrac{1}{r})$ is equivalent to $\delta_\lambda(x,y)<r$ which defines $B_{\delta_\lambda}(x,r)$.
\end{proof}

Let us point out that the actual construction of the sequence $\lambda(i)$ will depend only on $K$ itself. Hence the $\delta_\lambda$ balls are strictly provided only by $K$.

\section{The algorithm for the explicit computation of the sequences $\lambda$. The finite case}\label{sec:algorithm}

In this section  we consider the case of $X=\{1,2,\ldots,n\}$ for some large integer $n$. The kernel $K$ defined on $X\times X$ can be regarded as an $n\times n$ symmetric matrix with positive entries $K_{ij}$. Since each $K_{ij}$ is positive the hypothesis $b)$ in Lemma~\ref{lemma:sequenceLambda} holds trivially since $\Lambda_\infty\geq \min K_{ij}>0$. Instead hypothesis $a)$ in Lemma~\ref{lemma:sequenceLambda} holds if $K_{ii}=\sup_j K_{ij}$.

In order to construct sequences $\lambda$, and then $\delta_\lambda$, associated to this matrix $K$ we shall need to deal in the algorithm with the composition of neighborhoods of the diagonal.

Let $U$ and $V$ be two subsets of $\{1,2,\ldots,n\}^{2}=X\times X$. Then, as before $V\circ U = \{(i,k): (i,j)\in U \textrm{ and } (j,k)\in V \textrm{ for some } j=1,2,\ldots,n\}$.
\begin{proposition}\label{propo:productmatrixA}
For a given $U\subseteq \{1,2,\ldots,n\}^n$ set $A_U=(a_{ij}(U))$ to denote the $n\times n$ rest matrix defined by $a_{ij}(U)=1$ of $(i,j)\in U$ and $a_{ij}(U)=0$ otherwise. Then the set $V\circ U$ is given by the non vanishing entries of the product matrix $A_U A_V$. Precisely
\begin{equation*}
V\circ U =\left\{(i,j)\in\left\{1,\ldots,n\right\}^2: \sum_{k=1}^n a_{ik}(U) a_{kj}(V)\geq 1\right\}.
\end{equation*}
\end{proposition}
\begin{proof}
Notice that $\sum_{k=1}^n a_{ik}(U) a_{kj}(V)\geq 1$ if and only there exists $k\in\{1,\ldots,n\}$ such that $a_{ik}(U)=1$ and $a_{kj}(V)=1$. In other words, if and only if $(i,k)\in U$ and $(k,j)\in V$, as desired.
\end{proof}

The next result is important at showing when the iterated composition of a neighborhood of the diagonal finally covers the whole space $\{1,2,\ldots,n\}^2$.
\begin{lemma}
Let $U$ be a set in $\{1,2,\ldots,n\}^2$ such that $U$ contains the three main diagonals of $\{1,2,\ldots,n\}^2$. Precisely, $(i,i-1)$, $(i,i)$ and $(i,i+1)$ belong  to $U$ for every $i=1,2,\ldots,n$. Then there exists $m$ such that $U^{(m)}=\{1,2,\ldots,n\}^2$.
\end{lemma}
\begin{proof}
From the representation of $U$ in terms of the matrix $A_U$ and the current hypothesis in $U$ we have that the matrix $A_U$ has ones at least in the three main diagonals. In other words, $a_{i,j}\geq 0$, $a_{i,i}=a_{i-1,i}=a_{i,i+1}=1$. Then $A^2_{U}$ has positive values at least in the entries of the five diagonals $\triangle=\{(i,i): i=1,\ldots,n\}$, $\triangle^+_1=\{(i,i+1): i=1,\ldots,n-1\}$, $\triangle^{-}_1=\{(i-1,i):i=2,\ldots,n\}$, $\triangle^+_2=\{(i,i+2): i=1,\ldots,n-2\}$ and $\triangle^{-}_2=\{(i-2,i):i=3,\ldots,n\}$. Iteration of the above argument shows that the composition of $U$ becomes wider around the diagonal and after a finite number of compositions the set $\{1,\ldots,n\}^2$ is completely covered.
\end{proof}

We are now in position to describe the basic steps of an algorithm to find a sequence $\lambda(i)$ associate to the kernel $K$.

\bigskip
\noindent\textbf{Algorithm.} Let $K=(K_{ij})$ be a $n\times n$ symmetric matrix with positive entries.

\medskip
\noindent\textbf{Step 1.} Compute the minimum of the values of $K$ on
        the three main diagonals
        $\Lambda_0=\min\{K_{i-1,i}; K_{i,i};K_{i,i+1}: i=1,\ldots,n\}$,

\medskip
\noindent\textbf{Step 2.} Build the matrix
        $A_0=A_{\{(i,j): K_{ij}\geq \Lambda_0\}}$
        as in Proposition~\ref{propo:productmatrixA};

\medskip
\noindent\textbf{Step 3.} Compute $A_0^3$;

\medskip
\noindent\textbf{Step 4.} Define $U_0$ as the subset of those $(i,j)$ in $\{1,\ldots,n\}^2$ such that the entry in $(i,j)$ of $A_0^3$ is positive;

\medskip
\noindent\textbf{Step 5.} Find $\Lambda_1=\max\{\alpha:\{K\geq \alpha\}\supseteq U_0\}$;

\medskip
\noindent\textbf{Step 6.} Build the matrix $A_1 = A_{\{(i,j): K_{ij}\geq\Lambda_1\}}$ as in Proposition~\ref{propo:productmatrixA};

\medskip
\noindent\textbf{Step 7.} Compute $A_1^3$;

\medskip
\noindent\textbf{Step 8.} Define $U_1=\{(i,j): \textrm{ the entry } (i,j) \textrm{ of } A_1^3 \textrm{ is positive}\}$;

\medskip
\noindent\textbf{Step 9.} Find $\Lambda_2 = \max\{\alpha: \{K\geq \alpha\}\supseteq U_1\}$;

\medskip
\noindent\textbf{$\boldmath{\cdots}$}

\medskip
\noindent The iteration stops after a finite  number of steps so we get the sequence  $\Lambda_0, \Lambda_1, \ldots, \Lambda_k$. It is clear that $\Lambda_k<\Lambda_{k-1}<\cdots<\Lambda_2<\Lambda_1$. Without any  extra condition on $K$ it could happen that $\Lambda_0\leq\Lambda_1$. But if $\Lambda_0$ is larger than  all the entries of $K$ outside the three main diagonals we have
\begin{equation*}
\Lambda_k<\Lambda_{k-1}<\cdots<\Lambda_2<\Lambda_1<\Lambda_0
\end{equation*}

\medskip
\noindent\textbf{Step $k+1$.} Set $\lambda(i)=\Lambda_{k-i}$; $i=0,\ldots,k$;

\medskip
\noindent\textbf{Step $k+2$.} Compute a version of $\lambda^{-1}$;

\medskip
\noindent\textbf{Step $k+3$.} Define $\delta_\lambda(i,j)= 2^{-\lambda^{-1}(K_{ij})}$;

\medskip
\noindent\textbf{Step $k+4$.} Plot $\delta_\lambda$ balls
$B_{\delta_\lambda}(i,r) = \{j: K_{ij}>\lambda (\log_2\tfrac{1}{r})\}$ for $i$ fixed and $0<r<1$.

\bigskip

\noindent The script in \verb|Python| for this algorithm is the following.

\linespread{1}
\begin{lstlisting}[language=Python,caption= Algorithm in Python]%caption= Algorithm in Python
import numpy as np
import matplotlib.pyplot as plt
import networkx as nx

## Value of n
n=n
## Compute minimum of K
Kmin=np.amin(K)
## Compute Lambda_0
lambda_0=0

aux=np.zeros((n))
for i in range(n-1):
    ## Compare inner values main diagonals
    aux[i]=min(K[i,i],K[i,i+1])
## Compare the remaining values in the main diagonals
aux[n-1]=K[n-1,n-1]
lambda_0=min(aux)

## Define matrix A
A=np.zeros((n,n))

for i in range(n):
    for j in range(n):
        if K[i,j]>=lambda_0:
            A[i,j]=1

## Compute  B=A^3
B=(A.dot(A)).dot(A)

## Compute  Bpos
Bpos=np.zeros((n,n))
for i in range(n):
    for j in range(n):
        if B[i,j]>=1:
            Bpos[i,j]=1

## Compute C
C=K*Bpos

## Compute minimum of the positive values of C
auxC=np.max(K)
for i in range(n):
    for j in range(n):
        if C[i,j]>0:
            auxC=min(auxC,C[i,j])
lambda_1=auxC

## Iterate
## Variables

lambda_i=np.zeros((n))
lambda_i[0]=lambda_0
lambda_i[1]=lambda_1

A_i=np.zeros((n,n,n))
A_i[0,:,:]=A

B_i=np.zeros((n,n,n))
B_i[0,:,:]=B

Bpos_i=np.zeros((n,n,n))
Bpos_i[0,:,:]=Bpos

C_i=np.zeros((n,n,n))
C_i[0,:,:]=C

## While
h=1
while lambda_i[h]>Kmin:
    ## Define matrix A
    for i in range(n):
        for j in range(n):
            if K[i,j]>=lambda_i[h]:
                A_i[h,i,j]=1

    ## Compute B=A^3
    B_i[h,:,:]=(A_i[h,:,:].dot(A_i[h,:,:])).dot(A_i[h,:,:])

    ## Bpos
    for i in range(n):
        for j in range(n):
            if B_i[h,i,j]>=1:
                Bpos_i[h,i,j]=1

    ## Compute C
    C_i[h,:,:]=K*Bpos_i[h,:,:]

    ## Compute minimum of the positive values of C
    auxC=np.max(K)
    for i in range(n):
        for j in range(n):
            if C_i[h,i,j]>0:
                auxC=min(auxC,C_i[h,i,j])
    lambda_i[h+1]=auxC
    h+=1

## End while

## Rearranging Lambda
lambda_i=lambda_i[0:h+1]
lambda_i=lambda_i[::-1]

## Inverse function of Lambda
def lambda_funct_inv(t,lambd):
    if t<0:
        print ('t must be larger or equal to the minimum value of lambda')
    if 0<=t<lambd[0]:
        inv=0
    for kk in range(len(lambd)-1):
        if lambd[kk]<=t<lambd[kk+1]:
            inv=kk+1
    if t>=lambd[len(lambd)-1]:
        inv=len(lambd)
    return inv

## Compute the matrix
def dist_frink_inv(nodo1,nodo2):
    distFinv=2**(-lambda_funct_inv(K[nodo1,nodo2],lambda_i))
    return distFinv

dist_array_Finv=np.zeros((n, n))
for v in range(n):
    for w in range(n):
        dist_array_Finv[v,w]=dist_frink_inv(v,w)

## Construct the graph starting from K
G = nx.Graph()
G = nx.from_numpy_matrix(np.matrix(K))

## Plot the graph
layout = nx.spring_layout(G)

plt.figure()
plt.title('Graph')
node_color=np.ones(n)
nx.draw(G, layout, node_color=node_color,with_labels=False)
nx.draw_networkx_labels(G, layout, font_size=12, font_family='sans-serif')
plt.show()

## Drawing balls centered at i
for k in range(n):
    for v in range(h+1):
        if dist_array_F[i][k] > lambda_i[v]:
            node_color[k]=h-v

node_color[i]=h+1
\end{lstlisting}

\linespread{1.3}

\section{Test and comparison with the diffusive metric for Newtonian type affinities}\label{sec:TestComparison}

The results in \cite{AiGoAffinity} suggest testing the algorithm on affinities defined as discretizations of Newtonian type potentials of the form
\begin{equation*}
K_\alpha(x,y)=\frac{1}{\abs{x-y}^\alpha}
\end{equation*}
for $\alpha$ positive. Once a discretization of $K_\alpha$ is given we may run our algorithm and also the well known diffusion metric introduced in \cite{CoifmanLafon06}. See also \cite{Bronstein}. Let us recall that the diffusive metric at time $t>0$ is given by
\begin{equation*}
d_t(i,j) = \left\{\sum_l e^{2t\nu_l}\abs{x_i^l - x^l}^2\right\}^{\tfrac{1}{2}}
\end{equation*}
where $x^l$, $\nu_l$, $l=1,\ldots,L$ are the eigenvectors and the eigenvalues of the Laplace operator on the graph with affinity given by the metric $K_{ij}$.

We shall only  write down the comparison of the families of $\delta_\lambda$-balls, $d_t$-balls and Euclidean balls for a couple of values of the radio, when we consider the discretization
\begin{equation*}
K_{ij} = 
  \begin{cases}
    2, & \text{for } i=j \\
    \abs{i-j}^{-\alpha}, & \text{for } i\neq j
  \end{cases}
\end{equation*}
with $i, j=0,\ldots,59$.

It is worthy pointing out at here that the choice of $60$ points of discretization is only taken for the sake of getting better images for the graphs. In particular for the visibility of some edges.

Let us also point out that in the following graphs, the numerical label of each vertex is assigned according to the order of the rows in the affinity matrix, but a priori has nothing to do with distance or affinity.

Figure~\ref{fig:grafo60nodes} labels with the integers $0,1,\ldots,59$ the $60$ vertices of our graph.

\begin{figure}[h!]
\includegraphics[width=7.5cm]{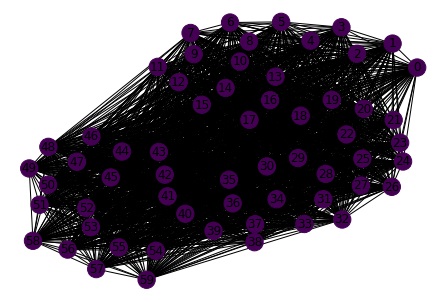}
\caption{Graph}\label{fig:grafo60nodes}
\end{figure}

We shall now plot some balls centered at two different vertices, $25$ and $50$, each for the three metrics, the Euclidean metric (E), the Diffusive metric (D) with $t= 0.005$ and Frink's metric (F). The comparison of both, (D) and (F) with the Euclidean (E) is essential because $K$ itself is built in terms of (E). Let us say again that we are interested in the shape of the balls but not in the particular radii for which those balls are attained. This fact is particulary clear in this case where the Euclidean metric is unbounded. Nevertheless we shall write out the values of the radii for which each ball in each metric is plotted. Actually the following pictures show in different colors the annuli between two consecutive balls. We use yellow for the center, green for the first annulus, turquoise for the second, lavender for the third and purple for the last annulus.

%
%

\begin{figure}[h]
    \begin{subfigure}[h]{0.45\textwidth}
\includegraphics[width=\textwidth]{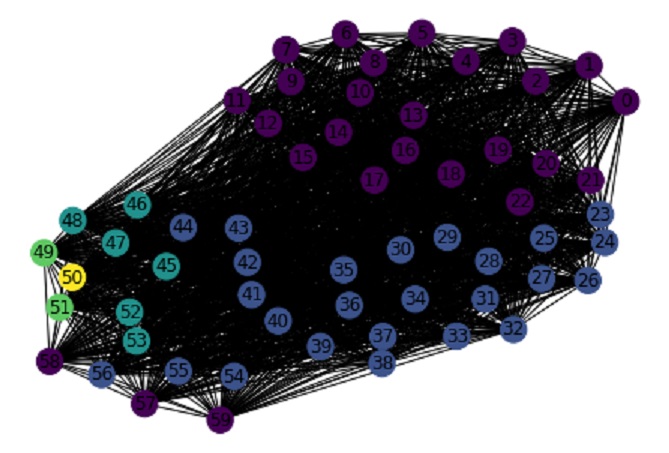}
\caption*{\textbf{(D)\,}\footnotesize{ Y, G, $0.11$, T, $0.135$, L, $0.31$, P, $0.404327$}}
\end{subfigure}
    \begin{subfigure}[h]{0.45\textwidth}
\includegraphics[width=\textwidth]{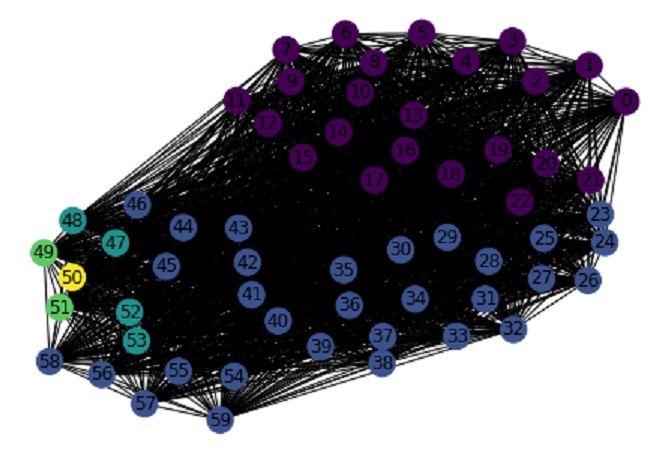}
\caption*{\textbf{(F)}\,\footnotesize Y, $0.0169492$, G, $0.037037$, T, $0.111111$, L, $0.333333$, P, $1$}
\end{subfigure}
%
\newline
    \begin{subfigure}[h]{0.45\textwidth}
\includegraphics[width=7.5cm]{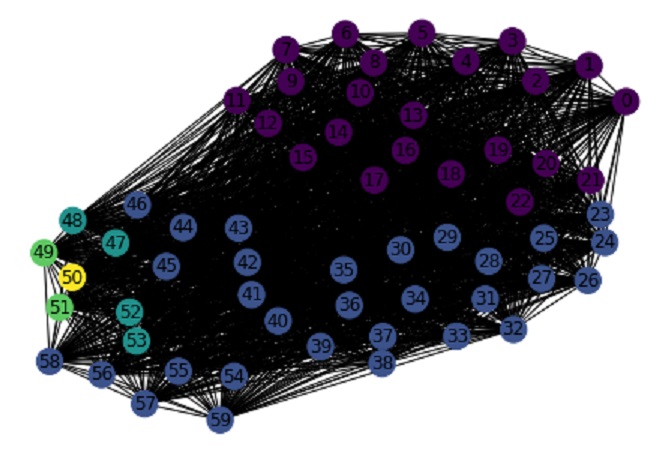}
\caption*{\textbf{(E)\,}\footnotesize Y, G, $1$, T, $3$, L, $27$, P, $59$}
\end{subfigure}
\caption{\textbf{Center at $50$}}\label{fig:BallsCenter50}
\end{figure}

In the Figure~\ref{fig:BallsCenter50} and Figure~\ref{fig:BallsCenter25} we use capital letters, $Y, G, T, L$ and $P$ for denote the colors. The sequences of letters and numbers describe the inner and outer radii of each annulus.

%
%
%
%
\begin{figure}[h]
\begin{subfigure}[h]{0.45\textwidth}
\includegraphics[width=\textwidth]{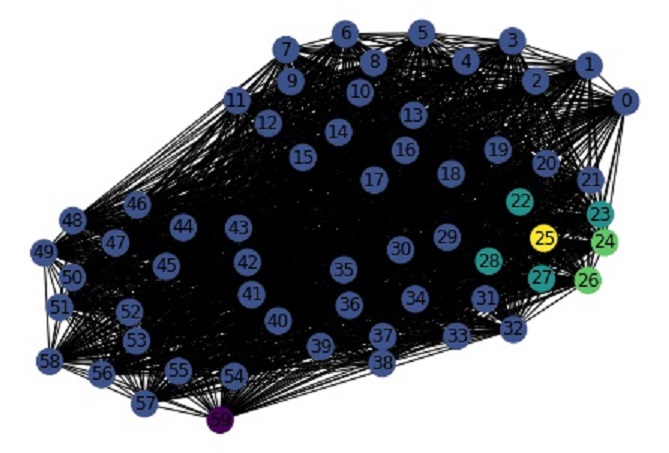}
\caption*{\textbf{(D)\,}\footnotesize Y, G, $0.13$, T, $0.17$, L, $0.212$, P, $0.404327$}
\end{subfigure}
\begin{subfigure}[h]{0.45\textwidth}
\includegraphics[width=\textwidth]{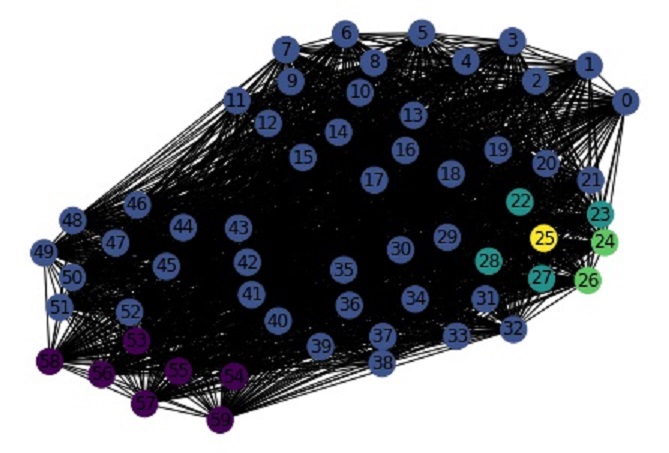}
\caption*{\textbf{(F)\,}\footnotesize Y, $0.0169492$, G, $0.037037$, T, $0.111111$, L, $0.333333$, P, $1$}
\end{subfigure}
\newline
\begin{subfigure}[h]{0.45\textwidth}
\includegraphics[width=8cm]{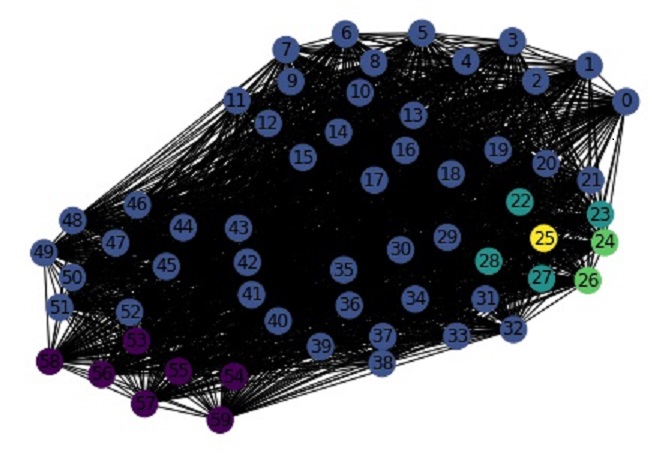}
\caption*{\textbf{(E)\,}\footnotesize Y, G, $1$, T, $3$, L, $27$, P, $59$}
\end{subfigure}
\caption{\textbf{Center at $25$}}\label{fig:BallsCenter25}
\end{figure}

It is worthy noticing that the sequence of raddi for (D) has been chosen in such a way that the $d_t$ balls become as close as possible to Euclidean balls. At least for this simple situation, of a kernel defined by a metric, the metrization scheme, (F), introduced here seems to reproduce the exact shapes of the balls associated to the metric defining the kernel. It could be argued that the exponential character of Frink's construction provides only a few balls of the graph. Nevertheless we know from the very proof of our  main result that we have at hand changing the initial parameter $\Lambda<\Lambda_\infty$ to produce a profuse diversity of sequences $\lambda(i)$. Another somehow arbitrary step of the algorithm is the use of the main three diagonal of our affinity matrix $K$. Starting with the main five diagonals will produce another family of F-balls and annuli.



\begin{thebibliography}{BBL{\etalchar{+}}17}
	
	\bibitem[AG18]{AiGoAffinity}
	Hugo Aimar and Ivana G\'{o}mez, \emph{Affinity and distance. {O}n the
		{N}ewtonian structure of some data kernels}, Anal. Geom. Metr. Spaces
	\textbf{6} (2018), 89--95. \MR{3816950}
	
	\bibitem[BBL{\etalchar{+}}17]{Bronstein}
	M.~M. Bronstein, J.~Bruna, Y.~LeCun, A.~Szlam, and P.~Vandergheynst,
	\emph{Geometric deep learning: going beyond euclidean data}, IEEE Signal
	Processing Magazine \textbf{34} (2017), no.~4, 18--42.
	
	\bibitem[Chi27]{Chi27}
	E.~W. Chittenden, \emph{On the metrization problem and related problems in the
		theory of abstract sets}, Bull. Amer. Math. Soc. \textbf{33} (1927), 13--34.
	
	\bibitem[CL06]{CoifmanLafon06}
	Ronald~R. Coifman and St\'ephane Lafon, \emph{Diffusion maps}, Appl. Comput.
	Harmon. Anal. \textbf{21} (2006), 5--30.
	
	\bibitem[Fri37]{Frink37}
	A.~H. Frink, \emph{Distance functions and the metrization problem}, Bull. Amer.
	Math. Soc. \textbf{43} (1937), no.~2, 133--142. \MR{1563501}
	
	\bibitem[Kel75]{Kelleybook75}
	John~L. Kelley, \emph{General topology}, Springer-Verlag, New York-Berlin,
	1975, Reprint of the 1955 edition [Van Nostrand, Toronto, Ont.], Graduate
	Texts in Mathematics, No. 27. \MR{0370454}
	
	\bibitem[MS79]{MaSe79Lip}
	Roberto~A. Mac\'{\i}as and Carlos Segovia, \emph{Lipschitz functions on spaces
		of homogeneous type}, Adv. in Math. \textbf{33} (1979), no.~3, 257--270.
	\MR{546295}
	
\end{thebibliography}

\newcommand{\etalchar}[1]{$^{#1}$}
\providecommand{\bysame}{\leavevmode\hbox to3em{\hrulefill}\thinspace}
\providecommand{\MR}{\relax\ifhmode\unskip\space\fi MR }
\providecommand{\MRhref}[2]{%
	\href{http://www.ams.org/mathscinet-getitem?mr=#1}{#2}
}
\providecommand{\href}[2]{#2}



\bigskip

\noindent{\footnotesize Mar\'ia Florencia Acosta, Hugo Aimar, and Ivana G\'omez.
	\textsc{Instituto de Matem\'{a}tica Aplicada del Litoral, CONICET, UNL.}
	
	%
%
\medskip

\noindent \textit{Address.} \textmd{IMAL, CCT CONICET Santa Fe, Predio ``Alberto Cassano'', Colectora Ruta Nac.~168 km~0, Paraje El Pozo, S3007ABA Santa Fe, Argentina.}
}

\bigskip
\end{document}